\documentclass{article}
\usepackage{amsfonts}
\usepackage{amssymb}
\usepackage{amsmath}
\usepackage{amsthm}
\usepackage{tikz}
\tikzset{vertex/.style={circle,fill,draw,scale=.2}}
\usepackage{graphicx}
\usepackage{xy} \xyoption{all}

\newtheorem{theorem}{Theorem}

\newtheorem{corollary}[theorem]{Corollary}

\newtheorem{example}[theorem]{Example}

\newtheorem{lemma}[theorem]{Lemma}

\newtheorem{proposition}[theorem]{Proposition}

\begin{document}

\title{On Pr\"{u}fer-like Properties of Leavitt Path Algebras}

\author{S. Esin, M. Kanuni, A. Ko\c{c}, K. Radler, K. M. Rangaswamy} 

\maketitle

\begin{abstract}
Pr\"{u}fer domains and subclasses of integral domains such as Dedekind domains
admit characterizations by means of the properties of their ideal lattices.
Interestingly, a Leavitt path algebra $L$, in spite of being non-commutative
and possessing plenty of zero divisors, seems to have its ideal lattices
possess the characterizing properties of these special domains. In [8] it was
shown that the ideals of $L$ satisfy the distributive law, a property of
Pr\"{u}fer domains and that $L$ is a multiplication ring, a property of Dedekind domains. In this paper, we first show that $L$ satisfies two more
characterizing properties of Pr\"{u}fer domains which are the ideal versions of
two theorems in Elementary Number Theory, namely, for positive integers
$a,b,c$, $\gcd(a,b)\cdot\operatorname{lcm}(a,b)=a\cdot b$ and $a\cdot
\operatorname{gcd}(b,c)=\operatorname{gcd}(ab,ac)$. We also show that $L$
satisfies a characterizing property of almost Dedekind domains in terms of the
ideals whose radicals are prime ideals. Finally, we give necessary and
sufficient conditions under which $L$ satisfies another important
characterizing property of almost Dedekind domains, namely the cancellative
property of its non-zero ideals.

\end{abstract}

\section{Introduction\protect\bigskip }

This paper is devoted to investigating some of the Pr\"{u}fer-like properties
of the ideals in a Leavitt path algebra $L:=L_{K}(E)$ of an arbitrary directed
graph over a field $K$. Recall that an integral domain $D$ is called a
\textit{valuation domain }if its ideals are totally ordered by set inclusion.
$D$ is called a Pr\"{u}fer \textit{domain} if all its localizations at maximal
ideals are valuation domains. $D$ is called an \textit{almost Dedekind domain}
if all its localizations at maximal ideals are noetherian valuation domains
and $D$ is called a \textit{Dedekind domain} if it is a noetherian domain and
all its localizations at maximal ideals are noetherian valuation domains
%Thus
%a Dedekind domain is an almost Dedekind domain and an almost Dedekind domain
%is a Pr\"{u}fer domain but not conversely 
(see \cite{Larsen}). There are many
equivalent characterizations of Pr\"{u}fer domains that are widely studied in
the literature.  Some of the characterizations of Pr\"{u}fer domains can be listed as given in \cite[Theorem 6.6]{Larsen}:
\begin{enumerate}
\item $R$ is a Pr\"{u}fer domain.
\item If $AB=AC,$ where $A,B,C$ are ideals of $R$ and $A$ is finitely
generated and nonzero, then $B=C.$
\item $A(B\cap C)=AB\cap AC$ for all ideals $A,B,C$ of $R.$
\item $(A+B)(A\cap B)=AB$ for all ideals $A,B$ of $R.$
\item $A\cap (B+C)=(A\cap B)+(A\cap C)$ for all ideals $A,B,C$ of $R.$
\item If $A$ and $C$ are ideals of $R,$ with $C$ finitely generated, and if $
A\subseteq C,$ then there is an ideal $B$ of $R$ such that $A=BC.$
\end{enumerate}

Although a Leavitt path algebra $L$ is non-commutative in nature and has
plenty of zero divisors, it is somewhat intriguing and certainly interesting
that the ideals of such a highly non-commutative algebra share many of the
properties of the ideals of various types of (commutative) integral domains.
To start with, the multiplication of ideals in $L$ is commutative (\cite{AAS},
\cite[Theorem 3.4]{Mulp-Ranga}), $L$ satisfies the property of a B\'{e}zout domain,
namely, all the finitely generated ideals of $L$ are principal (\cite{R-2}),
every ideal of $L$ is projective, a property of Dedekind domains and the ideal
lattice of $L$ is distributive (\cite{Mulp-Ranga}) which characterizes Pr\"{u}fer
domains among integral domains.\\\\
In this paper, we consider characterizing properties of Pr\"{u}fer domains
which are ideal versions of well-known theorems in elementary number theory.
Recall that if $a,b,c$ are positive integers, then $\gcd(a,b)\cdot
\operatorname{lcm}(a.b)=a\cdot b$ and $a\cdot\gcd(b,c)=\gcd(ab,ac)$ and
$a\cdot\operatorname{lcm}(b,c)=\operatorname{lcm}(ab,ac)$. The first two results when
expressed in terms of its ideals lead to a characterization of Pr\"{u}fer
domains: An integral domain $D$ is a Pr\"{u}fer domain if and only if, for any
two non-zero ideals $A,B$ of $D$, $(A\cap B)(A+B)=AB$ and, if and only if, for
any three ideals $A,B,C$ in $D$, we have $A(B\cap C)=AB\cap AC$ (see \cite{Larsen}).
We will show that every Leavitt path algebra satisfies these two
characterizing properties. (Note that the ideal version of the result $a.lcm(b,c) = lcm (ab,ac)$, namely, $A(B+C) = AB + AC$  holds for ideals $A$, $B$, $C$ in any ring $R$.) We also investigate whether a Leavitt path algebra
$L$ possesses any of the properties of almost Dedekind domains which form a
subclass of the class of Pr\"{u}fer domains. It is known (see \cite{Larsen}) that
an integral domain $D$ is an almost Dedekind domain if and only if every ideal
in $D$, whose radical is prime, is a power of a prime ideal. We show that every
Leavitt path algebra also possesses this characterizing property. As a corollary, we show that two more properties 
of a Dedekind domain are satisfied by Leavitt path algebras. Almost Dedekind domains $D$ are also characterized 
among integral domains by the property that every non-zero ideal $A$ of $D$ is cancellative, that is,
$AB=AC$ implies $B=C$ for any two ideals $B,C$ of $D$. While not every Leavitt
path algebra satisfies this property, we give necessary and sufficient
conditions on the graph $E$ under which every non-zero ideal of $L_{K}(E)$ is
cancellative. Various graphical constructions illustrate our results.\\

\section{Preliminaries}

In this section, we will mention some of the needed basic concepts and results in Leavitt path algebras. 
A \textit{directed graph} $E=(E^{0},E^{1},r,s)$ consists of two countable
sets $E^{0},E^{1}$ and functions $r,s:E^{1}\rightarrow E^{0}$. The elements $%
E^{0}$ and $E^{1}$ are called \textit{vertices} and \textit{edges},
respectively. For each $e\in E^{1}$, $s(e)$ is the source of $e$ and $r(e)$
is the range of $e.$ If $s(e)=v$ and $r(e)=w$, then we say that $v$ emits $e$
and that $w$ receives $e$. A vertex which does not receive any edges is
called a \textit{source,} and a vertex which emits no edges is called a 
\textit{sink}. A vertex $v$  is called a \textit{regular vertex} if it emits a non-empty finite set of edges. A vertex is called 
an \textit{infinite emitter} if it emits infinitely many edges.

 A graph is called \textit{\ row-finite} if $s^{-1}(v)$ is a
finite set for each vertex $v$. For a row-finite graph the edge set $E^{1}$
of $E~$is finite if its set of vertices $E^{0}$ is finite. Thus, a
row-finite graph is finite if $E^{0}$ is a finite set.

A path in a graph $E$ is a sequence of edges $\mu =e_{1}\ldots e_{n}$ such
that $r(e_{i})=s(e_{i+1})$ for $i=1,\ldots ,n-1.$ In such a case, $s(\mu
):=s(e_{1})$ is the \textit{source }of $\mu $ and $r(\mu ):=r(e_{n})$ is the 
\textit{range} of $\mu $, and $n$ is the \textit{length }of $\mu ,$ i.e., $
l(\mu )=n$. If $s(\mu )=r(\mu )$ and $s(e_{i})\neq s(e_{j})$ for every $i\neq j$, then $
\mu $ is called a \textit{cycle}.

An exit for a  path $\mu =e_{1}\ldots e_{n}$ is an edge $e$ such that $s(e)=s(e_i)$ for some $i$ and $e \neq e_i$. 
A graph $E$ is said to satisfy \textit{condition} (L) if every cycle in $E$ has an exit.

%For $n\geq 2,$ define $E^{n}$ to be the set of
%paths of length $n,$ and $E^{\ast }=\bigcup\limits_{n\geq 0}E^{n}$ the set
%of all finite paths. 

	A subset $ D $ of vertices is said to be {\it downward directed} if for any $ u, v \in D $, there exists a $ w\in D $ such that $ u \geq w $ and $ v \geq w $. A subset $H$ of $E^0$ is called {\it hereditary} if whenever $v \in H$ and $w \in E^0$ for which $v \geq w$, then $w \in H$. $H$ is {\it saturated} if whenever a regular vertex $v$ has the property that $\{r(e)| e \in E^1, s(e)=v\} \subseteq H$, then $v \in H$. 
%\textcolor{red}{	We denote $\mathscr{H}_E$ the set of those subsets of $E^0$ which are both hereditary and saturated.\\
%Let $E$ be an arbitrary graph and let $H$ be a hereditary subset of $E^0$. The \textit{quotient graph $E/H$ of $E$ by $H$} is  defined as follows:
%	$$(E/H)^0 = E^0 \backslash H  \text{ and } (E/H)^1 = \{e \in E^1 | r(e) \notin H \}.$$
%	The range  and source functions for $E/H$ are defined by restricting the range and source functions of $E$ to $(E/H)^1$.}

The path $K$-algebra over $E$ is defined as the free $K$-algebra $%
K[E^{0}\cup E^{1}]$ with the relations:

\begin{enumerate}
\item[(1)] $v_{i}v_{j}=\delta _{ij}v_{i}$ \ for every $v_{i},v_{j}\in E^{0}.$

\item[(2)] $e_{i}=e_{i}r(e_{i})=s(e_{i})e_{i}$ $\ $for every $e_{i}\in
E^{1}. $
\end{enumerate}

This algebra is denoted by $KE$. Given a graph $E,$ define the extended
graph of $E$ as the new graph $\widehat{E}=(E^{0},E^{1}\cup (E^{1})^{\ast
},r^{\prime },s^{\prime })$ where $(E^{1})^{\ast }=\{e_{i}^{\ast
}~|~e_{i}\in E^{1}\}$ and the functions $r^{\prime }$ and $s^{\prime }$ are
defined as 
\begin{equation*}
r^{\prime }|_{E^{1}}=r,~~~~s^{\prime }|_{E^{1}}=s,~~~~r^{\prime
}(e_{i}^{\ast })=s(e_{i})~~~~~~\text{and~~~~~}s^{\prime }(e_{i}^{\ast
})=r(e_{i}).
\end{equation*}
The Leavitt path algebra of $E$ with coefficients in $K$ is defined as the
path algebra over the extended graph $\widehat{E},$ with relations:

\begin{enumerate}
\item[(CK1)] $e_{i}^{\ast }e_{j}=\delta _{ij}r(e_{j})$ \ for every $e_{j}\in
E^{1}$ and $e_{i}^{\ast }\in (E^{1})^{\ast }.$

\item[(CK2)] $v_{i}=\sum_{\{e_{j}\in
E^{1}~|~s(e_{j})=v_{i}\}}e_{j}e_{j}^{\ast }$ \ for every regular vertex $v_{i}\in E^{0}$.
\end{enumerate}

This algebra is denoted by $L_{K}(E)$. The conditions (CK1) and (CK2) are
called the Cuntz-Krieger relations. \\

A ring $R$ is called \textit{a ring with local units,} if for every non-empty
finite subset $X$ of $R$, there is a non-zero idempotent $u\in R$ such that $ux=x=xu$
for all $x\in X$. When  $E^0$ is finite, $L_{K}(E)$ is a ring with unit element 
$\displaystyle 1= \sum_{v\in E^{0}}v$. Otherwise, $L_{K}(E)$ is not a unital ring, but is a ring with local units 
consisting of sums of distinct elements of $E^0$.

\bigskip
 A useful observation is that every element $a$ of $L_{K}(E)$ can be written as
 $a=%
 %TCIMACRO{\tsum \limits_{i=1}^{n}}%
 %BeginExpansion
 {\textstyle\sum\limits_{i=1}^{n}}
 %EndExpansion
 k_{i}\alpha_{i}\beta_{i}^{\ast}$, where $k_{i}\in K$, $\alpha_{i},\beta_{i}$
 are paths in $E$ and $n$ is a suitable integer. Moreover, $L_{K}(E)=%
 %TCIMACRO{\tbigoplus \limits_{v\in E^{0}}}%
 %BeginExpansion
 {\textstyle\bigoplus\limits_{v\in E^{0}}}
 %EndExpansion
 L_{K}(E)v=%
 %TCIMACRO{\tbigoplus \limits_{v\in E^{0}}}%
 %BeginExpansion
 {\textstyle\bigoplus\limits_{v\in E^{0}}}
 %EndExpansion
 vL_{K}(E).$ Another useful fact is that if $p^{\ast}q\neq0$, where $p,q$ are
 paths, then either $p=qr$ or $q=ps$ where $r,s$ are suitable paths in $E$.\\

\bigskip

One of the most important properties of Leavitt path algebras is that each $L_K(E)$ is a $\mathbb{Z}$-graded $K$-algebra. 
That is, 
$L_{K}(E)={\displaystyle\bigoplus\limits_{n\in\mathbb{Z}}}L_{n}$ 
induced by defining,  $\deg
(v)=0$ for all $v\in E^{0}$ and $\deg(e)=1$, $\deg(e^{\ast})=-1$ for all $e\in E^{1}$. Further,  the homogeneous component $L_{n}$ for each 
$n\in \mathbb{Z}$ is given by
\[L_{n}=\left\{{\textstyle\sum}
k_{i}\alpha_{i}\beta_{i}^{\ast}\in L:\text{ }l(\alpha_{i})-l(\beta
_{i})=n\right\} \mbox{ where } k_i \in K, \alpha_{i},\beta_{i} \in Path(E).
\]
An ideal $I$ of $L_{K}(E)$ is said to be a \textit{graded ideal} if 
$I={\displaystyle\bigoplus\limits_{n\in\mathbb{Z}}}(I\cap L_{n})$.\\\\

 We shall be using the following concepts and results from \cite{T}. A vertex $v$ is called a \textit{breaking vertex of a hereditary subset $H$} if $v$ belongs to the set
 \[
 B_H:=\{ v \in E^0 \backslash H | \text{ v is an infinite emitter and } 0 < | s^{-1} (v) \cap r^{-1} (E^0 \backslash H)| < \infty \}.
 \]
 In words, $B_H$ consists of those vertices of $E$ which are infinite emitters, which do not belong to $H$, and for which the ranges of the edges they emit are all, except for a finite (but nonzero) number, inside $H$. For $v \in B_H$, the element $v^H$ of $L_K(E)$ is defined by
 \[
 v^H:=v - \displaystyle\sum_{e\in s^{-1}(v) \cap r^{-1}(E^0 \backslash H)} ee^{\ast} .
 \]
 We note that any such $v^H$ is homogeneous of degree $0$ in the standard $\mathbb{Z}$-grading on $L_K(E)$.
 Given a hereditary saturated subset
 $H$ and a subset $S\subseteq B_{H}$, $(H,S)$ is called an \textit{admissible
 	pair.}  The ideal generated by $H\cup\{v^{H}:v\in
 S\}$ is denoted by $I(H,S)$ where $(H,S)$ is an admissible pair. 
 %Given a hereditary saturated subset $H$ and a subset  $S \subseteq B_H$, 
 The \textit{quotient graph $E\backslash (H,S)$ of $E$ by an admissible pair $(H,S)$} is defined as follows:
 $$(E\backslash (H,S))^0 = (E^0 \backslash H) \cup \{ v' | v \in B_H \backslash S \},$$	
 $$(E\backslash (H,S))^1 = \{e \in E^1 | r(e) \notin H\} \cup \{ e' | e \in E^1 \text{ and } r(e) \in B_H \backslash S \},$$
 and range and source maps in $E\backslash (H,S)$ are defined by extending the range and source maps in $E$ when appropriate, and in addition setting $s(e')=s(e)$ and $r(e')=r(e)'$. It was shown in \cite{T} that $L_K(E)/ I(H,S) \cong L_K(E\backslash (H,S))$. \\
  
\bigskip

Let $\Lambda$ be an arbitrary non-empty set. Given a ring
$R$, $M_{\Lambda}(R)$ denotes the ring of matrices with entries from $R$, all
but finitely many of which are non-zero and where the rows and columns are
indexed by elements of $\Lambda$. We will be using an important result about
the ideals of the ring $M_{\Lambda}(R)$. This result was proved in Theorem 3.1
of \cite{L} when $R$ is a ring with identity $1$ and when $\Lambda$ is finite.
We need this result for rings with local units and when $\Lambda$ is an
arbitrary non-empty set. As far as we know, this generalized statement has not
appeared in print and so we give the general statement and its proof.
Thus Proposition \ref{Ideals in Mtrix rings}(a) below is a generalization of
Theorem 3.1 of \cite{L} to rings with local units. We thank Zak Mesyan for
help in writing the proof of Proposition \ref{Ideals in Mtrix rings}(b). For
rings $R$ with identity, this Proposition is also obtained by using the Morita
equivalence of $R$ and $M_{\Lambda}(R)$ (see \cite{AF}).\\\
\begin{proposition}
\label{Ideals in Mtrix rings} Suppose $R$ is a ring with local
units and $\Lambda$ is an arbitrary non-empty set.\\\\
(a) Every ideal of $M_{\Lambda}(R)$ is of the form $M_{\Lambda}(A)$ for some
ideal $A$ of $R$. The map $A\longmapsto M_{\Lambda}(A)$ defines a lattice
isomorphism between the lattice of ideals of $R$ and the lattice of ideals of
$M_{\Lambda}(R)$.\\\\
(b) For any two ideals $A,B$ of $R$, $M_{\Lambda}(AB)=M_{\Lambda}%
(A)M_{\Lambda}(B)$.
\end{proposition}

\begin{proof} If $A$ is an ideal of $R$, then it is easy to see that
$M_{\Lambda}(A)$ is an ideal of $M_{\Lambda}(R)$. Also, if the ideals $A\neq
B$, then $M_{\Lambda}(A)$ $\neq$ $M_{\Lambda}(B)$. Let $I$ be a non-zero ideal
of $M_{\Lambda}(R)$. We wish to show that $I=M_{\Lambda}(A)$ for some ideal
$A$ of $R$. Let $A$ be the set of all\ the entries at the (first row - first
column) position in all the matrices belonging to $I$. $A$ is clearly an ideal
in $R$. We wish to show that $I=M_{\Lambda}(A)$. Let $U$ denote the set of all
local units in $R$. Let $0\neq M=(m_{ij})\in I$. Corresponding to all the
finitely many non-zero entries $m_{ij}$ in $M$, choose a local unit $u\in U$
satisfying $um_{ij}=m_{ij}=m_{ij}u$ for all $i,j$. For any $k,l$, we have the
identity%
\[
E_{ij}uME_{kl}u=m_{jk}E_{il}u\qquad\qquad\qquad\qquad(\ast)
\]
where, for every $i,j$, $E_{ij}u$ denotes the $\Lambda\times\Lambda$ matrix
having $u$ at the $i$th row and $j$th column entry and $0$ at every other
entry. \ In particular, $E_{1j}uME_{k1}u=m_{jk}E_{11}u\in I$ and thus $m_{jk}\in A$
for all $j,k\in\Lambda$. Hence $M\in M_{\Lambda}(A)$. Conversely, let
$N=(a_{ij})\in M_{\Lambda}(A)$. Since $a_{il}\in A$, there is a matrix
$M=(m_{ij})\in I$ such that $m_{11}=a_{il}$. Let $v\in U$ be a local unit
satisfying $va_{ij}=a_{ij}=a_{ij}v$ $\ $and also $vm_{ij}=m_{ij}=m_{ij}v$ for
all the entries $a_{ij}$ in $N$. It is enough to show that $a_{il}E_{il}v\in
I$ for all $i,l$. Applying the identity $(\ast)$, we have
\[
a_{il}E_{il}v=m_{11}E_{il}v=E_{i1}vME_{1l}\in I\text{, for all }i,l\text{. }%
\]
This proves that $I=M_{\Lambda}(A)$.\\\\
It is straightforward to verify that, for any two ideals $A,B$ of $R$,
$M_{\Lambda}(A+B)=M_{\Lambda}(A)+M_{\Lambda}(B)$ and $M_{\Lambda}(A\cap
B)=M_{\Lambda}(A)\cap M_{\Lambda}(B)$. This shows that map $A\longmapsto
M_{\Lambda}(A)$ defines a lattice isomorphism.\\\\
We next show that $M_{\Lambda}(AB)=M_{\Lambda}(A)M_{\Lambda}(B)$ for any two
ideals $A,B$ of $R$. Given $I\in M_{\Lambda}(A)$ and $J\in M_{\Lambda}(B)$,
every entry of $IJ$ is a finite sum of the form $\sum a_{i}b_{j}$ for some
$a_{i}\in A$ and $b_{j}\in B$ and hence an element of $AB$. Thus $M_{\Lambda
}(A)M_{\Lambda}(B)$ $\subseteq M_{\Lambda}(AB)$.\\\\
To prove the reverse inclusion, first note that every element $M\in
M_{\Lambda}(AB)$ can be written as a finite sum of elements of the form
$abE_{ij}u$ where $a\in A,b\in B$ and $u$ is a local unit corresponding to the
finitely many non-zero entries of $M$ and also satisfying $ua=a=au$ and
$ub=b=bu$. Since $M_{\Lambda}(A)M_{\Lambda}(B)$ is an ideal of $M_{\Lambda
}(R)$, it is enough to show that $abE_{ij}u\in M_{\Lambda}(A)M_{\Lambda}(B)$,
for all $a\in A,b\in B$ and $i,j\in\Lambda$. Now $aE_{il}u\in M_{\Lambda}(A)$
and $bE_{lj}u\in M_{\Lambda}(B)$ and so%
\[
abE_{ij}u=abE_{il}uE_{lj}u=(aE_{il}u)(bE_{lj}u)\in M_{\Lambda}(A)M_{\Lambda
}(B)\text{.}%
\]
Thus $M_{\Lambda}(AB)=M_{\Lambda}(A)M_{\Lambda}(B)$ for any two ideals $A,B$
of $R$.
\end{proof}

Throughout the following, $L$ will denote the Leavitt path algebra
$L_{K}(E)$ of an arbitrary graph $E$ over a field $K$.

\section{Pr\"{u}fer-like properties satisfied in Leavitt path algebras }

In this section, we shall describe how the ideals of every Leavitt path
algebra $L$ satisfy two of the characterizing properties of a Pr\"{u}fer
domain mentioned in the introduction. In this connection, the graded ideals of
$L$ seem to be well-behaved and some extra efforts are needed in dealing with
the non-graded ideals of $L$. The following theorem consists of the results in
\cite{R-2} and \cite{R-1} which will be used in the sequel.

\begin{theorem}\label{reqthm}
\label{R-1-2} Let $I$ be a non-graded ideal of $%
L=L_{K}(E)$ with $H=I\cap E^{0}$ and $S=\{u\in B_{H}:~u^{H}\in I\}.$ Then

\begin{description}
\item[(i) \ ] (\cite[Theorem 4]{R-2}) $I=I(H,S)+\sum_{t\in T} \langle f_{t}(c_{t}) \rangle$ where $T$ is
some index set, for each $t\in T,$ $c_{t}$ is a cycle without exits in $%
E\backslash (H,S),$ $c_{t}^{0}\cap c_{s}^{0}=\emptyset $ for $t\neq s$ and $%
f_{t}(x)\in K[x]$ is a polynomial with its constant term non-zero and is of
the smallest degree such that $f_{t}(c_{t})\in I.$

\item[(ii)] (\cite[Lemma 3.6]{R-1}) $I(H,S)$ is the largest graded ideal
inside $I.$
\end{description}
\end{theorem}

%Recall, that every non-graded ideal in a Leavitt path algebra $L_{K}(E)$ can be written as $I=I(H,S)+\sum_{t\in T}\langle f_{t}(c_{t})\rangle $ where $ H=I\cap E^{0}$ and $S=\{u\in B_{H}:u^{H}\in I\}$ and $T$ is some index set, where $c_{t}$ is a cycle without exits in $E\backslash (H,S)$, $c_{t}^{0}\cap c_{s}^{0}=\emptyset $ for $t\neq s$ and $f_{t}(x)\in K[x]$ is a polynomial of the smallest degree such that $f_{t}(c_{t})\in I$.

\bigskip

We shall denote $I(H,S)$ by $gr(I)$ and call it the graded part of the ideal $I$. \\

Before proving  the main theorem, we consider the case of graded ideals which
are easy to handle. A useful property of graded ideals of a Leavitt path
algebra $L$ (see \cite[Lemma 3.1]{Mulp-Ranga}) is that if $A$ is a graded ideal of
$L$, then for any ideal $B$, $AB=A\cap B$.

\begin{lemma}\label{lemma2}
Let $A,B,C$ be three ideals of a Leavitt path algebra $L$.
If one of them is a graded ideal then%
\[
A(B\cap C)=AB\cap AC\text{.}%
\]
\end{lemma}
\begin{proof}
Case 1: Suppose $A$ is a graded ideal. Then 
by \cite[Lemma 3.1 (i)]{Mulp-Ranga}, 
\begin{eqnarray*}
A(B\cap C) &=&A\cap (B\cap C) \\
&=&(A\cap B)\cap (A\cap C) \\
&=&AB\cap AC
\end{eqnarray*}

Case 2: Suppose $B$ or $C$ is a graded ideal, say, $B$ is a graded ideal. Then 
\begin{eqnarray}
AB\cap AC &=&(A\cap B)\cap AC\text{ \ \ \ \ \ \ \ since }B\text{ is graded} \notag
\\
&=&B\cap AC\text{ \ \ \ \ \ \ \ \ \ \ \ \ \ \ \ since }AC\subseteq A \notag\\
&=&B(AC)\text{ \ \ \ \ \ \ \ \ \ \ \ \ \ \ \ \ since }B\text{ is graded} \notag \\
&=&ABC\text{ \ \ \ \ \ \ \ \ \ \ \ \ \ \ \ \ \ \ since } AB=BA \text{ by \cite[Theorem 3.4]{Mulp-Ranga}} \notag \\
&=&A(B\cap C)\text{ \ \ \ \ \ \ \ \ \ \ \ \ \ since }B\text{ is graded.} \notag
\end{eqnarray}
\end{proof}

Next, we consider the case when all ideals are non-graded in the next two lemmas. In the proofs, we shall be using Theorem 4.3 of \cite{Mulp-Ranga}, namely, $A\cap (B+C) = (A \cap B) + (A \cap C)$ for any three ideals $A$, $B$, $C$ in $L$. We shall also using the fact that for a graded ideal $I$, $IJ = I \cap J$ for any ideal $J$ in $L$.

\begin{lemma}
\label{NonGraded1}Let $A,B,$ and $C$ be non-graded ideals. If 
$A\subseteq gr(A)+gr(B\cap C)$ \ or ($B\subseteq gr(A)+gr(B\cap C)$ and 
$C\subseteq gr(A)+gr(B\cap C)$), then $A(B\cap C)=AB\cap AC$.
\end{lemma}

\begin{proof}
We want to show that $AB\cap AC\subseteq A(B\cap C)$ since the other
inclusion is always true.

Suppose that $A\subseteq gr(A)+gr(B\cap C).$ By the Modular Law, 
\begin{equation*}
A=\ A\cap (gr(A)+gr(B\cap C))=gr(A)+(A\cap gr(B\cap C)).
\end{equation*}
Then%
\begin{eqnarray*}
AB &=&(gr(A)+(A\cap gr(B\cap C)))B \\
&=&gr(A)B+(A\cap gr(B\cap C))B \\
&=&gr(A)B+(A\cap gr(B\cap C))\text{\ \ \ \ \ \ \ \ by \cite[Lemma 3.2]{Mulp-Ranga}}
\end{eqnarray*}%
and similarly, $AC=gr(A)C+(A\cap gr(B\cap C)).$ Hence 
\begin{equation*}
AB\cap AC=(gr(A)B+(A\cap gr(B\cap C)))\cap (gr(A)C+(A\cap gr(B\cap C)))
\end{equation*}%
By \cite[Theorem 4.3]{Mulp-Ranga}, 
\begin{eqnarray*}
AB\cap AC &=&[(gr(A)B+(A\cap gr(B\cap C)))\cap gr(A)C] \\
&&+[(gr(A)B+(A\cap gr(B\cap C)))\cap (A\cap gr(B\cap C)))]
\end{eqnarray*}%
Now, by \cite[Theorem 4.3, Lemma 3.1 (i)]{Mulp-Ranga}, 
\begin{eqnarray*}
AB\cap AC &=&gr(A)\cap (B\cap C)+gr(A)\cap gr(B\cap C)\cap C \\
&&+A\cap gr(B\cap C) \\
&=&gr(A)(B\cap C)+gr(A)gr(B\cap C) \\
&&+Agr(B\cap C) \text{ \ \ \ \ \ \ \ \ \ \ \ \ \ by using \cite[Lemma 3.1 (i)]{Mulp-Ranga} }\\
&\subseteq &A(B\cap C)
\end{eqnarray*}%
Now, suppose that $B\subseteq gr(A)+gr(B\cap C)$ and $C\subseteq
gr(A)+gr(B\cap C).$ Then, by the Modular Law, $B=$\ $B\cap (gr(A)+gr(B\cap
C))=gr(B\cap C)+(gr(A)\cap B)$ and similarly, $C=gr(B\cap C)+(gr(A)\cap C).$

Hence,%
\begin{eqnarray*}
AB &=&Agr(B\cap C)+A(gr(A)\cap B) \\
&=&Agr(B\cap C)+(gr(A)\cap B)\text{\ \ \ \ \  \ \ \ \ \ \ \ \ \ \ \ \ \ by \cite[Lemma 3.2]{Mulp-Ranga}}
\end{eqnarray*}%
and similarly, $AC=Agr(B\cap C)+(gr(A)\cap C).$\\

\noindent
Therefore, by using \cite[Theorem 4.3, Lemma 3.1 (i)]{Mulp-Ranga}, 
\begin{eqnarray*}
AB\cap AC &=&[ Agr(B\cap C)+(gr(A)\cap B )] \cap [ Agr(B\cap C)+(gr(A)\cap C )] \\
&=&Agr(B\cap C)+gr(A)gr(B\cap C)+gr(A)gr(B\cap C)+gr(A)(B\cap C) \\
&\subseteq &A(B\cap C)
\end{eqnarray*}
\end{proof}

In proving the next lemma, we shall be using the easy-to-see statement that,
for any two ideals $B,C$ of $L$,  $gr(B\cap C)=gr(B)\cap gr(C)$.
\bigskip

\begin{lemma}
\label{NonGraded2}Let $A,B,$ and $C$ be non-graded ideals of $L$. If $%
A\nsubseteq gr(A)+gr(B\cap C)$ \ and ($B\nsubseteq gr(A)+gr(B\cap C)$ or $%
C\nsubseteq gr(A)+gr(B\cap C)$), then $A(B\cap C)=AB\cap AC$.
\end{lemma}

\begin{proof}
Without loss of generality, we may assume $A\nsubseteq gr(A)+gr(B\cap C)$ and $B\nsubseteq
gr(A)+gr(B\cap C)$. Let $I=I(H,S)=gr(A)+gr(B\cap C)$. By Theorem \ref{reqthm} (i), %(\cite[Theorem 2.2]{Mulp-Ranga}),
\begin{eqnarray*}
A &=&I(H_{1},S_{1})+ \sum_{i\in X}\left\langle f_{i}(c_{i})\right\rangle
,~~B=I(H_{2},S_{2})+ \sum_{j\in Y}\left\langle g_{j}(c_{j})\right\rangle ,~
\text{\ and } \\
C &=&I(H_{3},S_{3})+ \sum_{k\in Z}\left\langle h_{k}(c_{k})\right\rangle
\end{eqnarray*}%
where $X,Y,$ and $Z$ are some index sets, $I(H_{1},S_{1})=gr(A),$ $%
I(H_{2},S_{2})=gr(B),$ $I(H_{3},S_{3})=gr(C)$ and for all $i\in X,~j\in Y,$
and $k\in Z,$ $f_{i}(x),g_{j}(x),h_{k}(x)\in K[x]$ and $c_{i},c_{j},$ and 
$c_{k}$ are cycles without exits in $E\backslash (H_{1},S_{1}),~E\backslash
(H_{2},S_{2}),$ and $E\backslash (H_{3},S_{3}),$ respectively. In $\overline{
L}=L/I\cong L_{K}(E\backslash (H,S)),$ $\overline{A}=(A+I)/I$ is an
epimorphic image of $A/gr(A)$ and let $~\overline{B}=(B+I)/I$ and $\overline{C}
=(C+I)/I.$ Hence, 
\begin{eqnarray*}
\overline{A} &=&\sum\limits_{i\in X^{\prime }}{}\left\langle
f_{i}(c_{i})\right\rangle ,~\overline{B}=(gr(A)+gr(B)+I)/I+\left[ \sum_{j\in
Y^{\prime }}{}\left\langle g_{j}(c_{j})\right\rangle +I\right] /I,\text{ and 
} \\
\overline{C} &=&(gr(A)+gr(C)+ I)/I+\left[ \sum_{k\in Z^{\prime }}{}\left\langle
h_{k}(c_{k})\right\rangle +I\right] /I,
\end{eqnarray*}%
where $X'$, $Y'$, $Z'$ are subsets of the sets $X$, $Y$, $Z$ respectively.

\bigskip
\noindent 
For the sake of convenience, we shall write \\

$~\overline{B}=(gr(A)+gr(B)+I)/I+\left[ \sum_2+I\right]/I, \textit{ where } \sum_2=\sum_{j\in
Y^{\prime }}{} \langle g_{j}(c_{j}) \rangle \text{ and } $

\bigskip

$\overline{C} =(gr(A)+gr(C)+I)/I+\left[ \sum_3 +I\right]/I,   \textit{ where } \sum_3=\sum_{k\in Z^{\prime }}{}\langle
h_{k}(c_{k}) \rangle $.\\\\

\bigskip

$ \qquad \qquad \qquad \overline{B}\cap \overline{C} =\left\{ (gr(A)+gr(B)+I)/I+\left[\sum_2 +I\right] /I\right\} $

\bigskip

$\qquad \qquad \qquad \qquad \qquad \cap \left\{ (gr(A)+gr(C)+I)/I+\left[ \sum_3 +I \right] /I \right\} $\\

\bigskip
$\qquad \qquad \qquad \qquad \qquad =(gr(A)+gr(B)+I)/I\cap \left[ \sum_3 +I \right]  /I $\\

\bigskip
$\qquad \qquad \qquad  \qquad \qquad +(gr(A)+gr(C)+I)/I\cap \left[ \sum_2 +I\right] /I $\\

\bigskip
$\qquad \qquad \qquad  \qquad \qquad +\left[ \sum_2 +I \right] /I\cap \left[ \sum_3 +I\right] /I,$\\

\noindent 
noting that $[(gr(A)+gr(B)+I)/I] \cap [(gr(A)+gr(C)+I)/I$ simplifies to $\overline{0}$.

\bigskip

\bigskip
$\qquad \qquad \bar{A}(\bar{B}\cap\bar{C})=\left[  \bar{A}(gr(A)+gr(B)+I)/I\right]
\cap\left[  \bar{A}\sum_{3}+I\right]  /I$

\bigskip
$\qquad \qquad \qquad  \qquad \qquad +\left[  \bar{A}(gr(A)+gr(C)+I)/I\right]  \cap\left[  \bar{A}\sum_{2}+I\right]
/I$

\bigskip
$\qquad \qquad \qquad  \qquad \qquad + \bar{A}\left( \left[ \left( \sum_3 + I \right)/I \right] \cap \left[\left(\sum_2 + I\right)/I \right] \right)   \quad \qquad (1)$\\
On the other hand,

\bigskip
$\qquad \qquad \qquad  \bar{A}\bar{B}=\left[  \bar{A}(gr(A)+gr(B)+I)/I\right]  +\left[  \bar
{A}\sum_{2}+I\right]  /I$

\bigskip

\noindent
and

\bigskip

$\qquad \qquad \qquad  \bar{A}\bar{C}=\left[  \bar{A}(gr(A)+gr(C)+I)/I\right]  +\left[  \bar{A}\sum
_{3}+I\right]  /I$

\bigskip

\noindent
Hence, by \cite[Theorem 4.3]{Mulp-Ranga}%

\bigskip

$\quad  \overline{A}~\overline{B}\cap \overline{A}~\overline{C} \> =\> \underset{=0\text{\ since }(gr(B)+gr(A)+I)/I\cap (gr(C)+gr(A)+I)/I=\overline{0}}{\underbrace{%
\left[  \bar{A}(gr(A)+gr(B)+I)/I\right]  \cap \left[  \bar{A}(gr(A)+gr(C)+I)/I\right] }}$

\bigskip

$\qquad \qquad \qquad  \quad  +\bar{A}(gr(A)+gr(B)+I)/I\cap \left[  \bar{A}\sum_{3}+I\right]  /I $

\bigskip

$\qquad \qquad \qquad  \quad  +  \left[  \bar{A}\sum_{2}+I\right]  /I   \cap \bar{A}(gr(A)+gr(C)+I)/I $
 
 \bigskip

$\qquad \qquad \qquad  \quad +\left[  \bar{A}\sum_{2}+I\right]  /I  \cap \left[  \bar{A}\sum_{3}+I\right]  /I \qquad \quad (2) $

\bigskip

\bigskip

\noindent
We wish to show that $\bar{A} \bar{B} \cap \bar{A} \bar{C} = \bar{A}( \bar{B} \cap \bar{C})$. Now comparing (1) and (2), all we need is to show that

\bigskip
\noindent
 $\quad \qquad  \bar{A} \left(\left[\left(\sum_2 + I\right)/I\right] \cap \left[\left(\sum_3+I\right)/I\right]\right) = \left[\bar{A} \left(\sum_2 + I\right)/I \right] \cap \left[\bar{A} \left(\sum_3+ I\right)/I\right].$
 
 \bigskip
 \noindent
Let $G$ be the graded ideal of $\bar{L}$ generated by the vertices on all the cycles $c_i$ where $i$ belongs to $X'$. Since the $c_i$ are cycles without exits in $E \backslash (H,S)$, $G$ is isomorphic to the ring direct sum $\bigoplus_{i\in X^{\prime
}}M_{\Lambda _{i}}(K[x,x^{-1}])$ where the $\Lambda _{i}$ are suitable index
sets by \cite[Theorem 2.7.3]{AAS}. Note that $\bar{A}$ is contained in $G$ and so $\bar{A} G = \bar{A}$, by \cite[Lemma 3.2]{Mulp-Ranga}. Then 

\bigskip
$\bar{A}([(\sum_2 + I)/I] \cap [(\sum_3 +I)/I]) =  \bar{A} G([\sum_2 +I)/I ] \cap [(\sum_3 +I)/I]) $

\bigskip 
$ \qquad \qquad \qquad \qquad   = \bar {A} (G[(\sum_2 + I)/I] \cap [G(\sum_3 +I)/I]) \text{ by Lemma 3} $

\bigskip 
$ \qquad \qquad \qquad \qquad   = \bar{A}([G \cap (\sum_2  + I)/I] \cap [G \cap (\sum_3 + I)/I]) \text{ as $G$ is graded.}$\\

Now all three ideals in the preceding equation are ideals of $G$ and G is isomorphic to the ring direct
sum  $\bigoplus_{i\in X^{\prime }}M_{\Lambda _{i}}(K[x,x^{-1}])$. Moreover, by Proposition 1, there is an isomorphism between the ideal lattices of $M_{\Lambda_i}(K[x, x^{-1}])$ and $K[x, x^{-1}]$ which preserves multiplication. Since $K[x,x^{-1}]$ is a Pr\"{u}fer domain, $H(K \cap L) = HK \cap HL$ holds for any three ideals of $K[x, x^{-1}]$ and consequently, any three ideals of $G$ also satisfy this property. We observe that 

\bigskip
\noindent
$\bar{A}([G \cap (\sum_2 + I)/I] \cap [G \cap (\sum_3 + I)/I])  = (\bar{A}[G \cap (\sum_2 + I)/I] \cap \bar{A}[G \cap (\sum_3 + I)/I]) $

\bigskip 
$ \qquad \qquad \qquad \qquad \qquad  = \bar{A}G(\sum_2 + I)/ I \cap \bar{A} G(\sum_3 + I)/I \text{ as $G$ is graded} $

\bigskip 
$ \qquad \qquad \qquad \qquad \qquad = \bar{A}(\sum_2 + I)/I \cap \bar{A}( \sum_3 + I)/I, \text{ by \cite[Lemma 3.2]{Mulp-Ranga}}.$

We thus conclude that $\bar{A}\bar{B} \cap \bar{A} \bar{C} = \bar{A}(\bar{B} \cap \bar{C})$.  Then 
$AB\cap AC=A(B\cap C)+(gr(A)+gr(B\cap C))$.  Now $A\cap B\cap C$ contains
both $AB\cap AC$ and $A(B\cap C)$ and so using modular law, we have%
\begin{eqnarray*}
AB\cap AC &=&(AB\cap AC)\cap (A\cap B\cap C) \\
&=&[A(B\cap C)+(gr(A)+gr(B\cap C))]\cap (A\cap B\cap C) \\
&=&A(B\cap C)+\underset{gr(A)\cap (A\cap B\cap C)+gr(B\cap C)\cap (A\cap
B\cap C)}{\underbrace{(gr(A)+gr(B\cap C))\cap (A\cap B\cap C)}} \\
&=&A(B\cap C)+gr(A)(B\cap C)+Agr(B\cap C) \\
&\subseteq &A(B\cap C).
\end{eqnarray*}%
Since $A(B\cap C)\subseteq AB\cap AC$ is always true, we get $A(B\cap
C)=AB\cap AC$.
\end{proof}

\noindent
Hence, we can state the main result.  
\bigskip

\begin{theorem}\label{MainThm}
If $A,B,C$ are any ideals of a Leavitt path algebra $L$ of an
arbitrary graph $E$, then 
\begin{equation*}
A(B\cap C)=AB\cap AC.
\end{equation*}
\end{theorem}
\begin{proof}
By using Lemmas \ref{lemma2}, \ref{NonGraded1} and \ref{NonGraded2}, the result follows.
\end{proof}

\bigskip

In elementary number theory, it is well-known that for any two positive
integers $a,b$, we have $\gcd(a,b) \cdot \text{lcm}(a,b) = ab$. This property can be stated for ideals as: for any ideals 
$A,B$, $(A+B)(A \cap B) = AB$. This equality holds for ideals in a Dedekind domain. If this equality holds for 
finitely generated ideals $A,B$, then the integral domain is a Pr\"{u}fer domain. 
The next theorem shows that any Leavitt path algebra satisfies this characterizing property. 
We shall prove this by using Theorem \ref{MainThm}.
%\begin{lemma}
%If $A$ or $B$ is graded, then $(A+B)(A\cap B)=AB.$
%\end{lemma}

%\begin{proof}
%$(A+B)(A\cap B)\subseteq AB$ is always true. Now, suppose $B$ is graded.
%Take $x\in AB.$ Then $x\in A$ and $x\in B,$ and so $x\in A\cap B.$ Since $B$
%is graded, there exist local units $u\in B\subseteq A+B$ with $x=ux=xu$.
%Hence $x=ux\in (A+B)(A\cap B).$
%\end{proof}

%We now show how ideals $A,B$ of any Leavitt path algebra also satisfy $(A+B)(A \cap B) = AB$ as a consequence of the main result.
\begin{theorem}
\label{5-6}For any two ideals $A$, $B$ of a Leavitt path algebra $L$,  $$(A+B)(A\cap B)=AB.$$
\end{theorem}

\begin{proof}
Now, by Theorem \ref{MainThm} and  \cite[Theorem 3.4]{Mulp-Ranga},
 \begin{eqnarray*}
(A+B)(A\cap B) &=&(A+B)A\cap (A+B)B \\
&\supseteq &BA\cap AB=AB\cap AB=AB.
\end{eqnarray*}\noindent
The converse inclusion is always true since 
$$(A+B)(A\cap B) =A(A\cap B)+B(A\cap B) \qquad \qquad \qquad $$ 
$$\qquad \qquad \qquad \qquad \quad \quad \> \subseteq  AB+BA=AB+AB=AB \qquad  \text{ by \cite[Theorem 3.4]{Mulp-Ranga}}.$$
Thus we obtain $(A+B)(A\cap B)=AB.$
\end{proof}
%%%%%%
%%%%%%
%%%%%%

%%%%%%
%%%%%%
%%%%%%

%%%%%%%%%
%%%%%%
%%%%%%

%%%%%%
%%%%%%
%%%%%%
\section{Almost Dedekind domains and Leavitt path algebras}
As noted in the Introduction, an almost Dedekind domain $D$ is a Pr\"{u}fer
domain with the property that all its localizations with respect to maximal ideals
are noetherian valuation domains. In this section, we investigate whether a
Leavitt path algebra $L$ satisfies any of the other characterizing properties of
almost Dedekind domains. Recall that  the \textit{radical} $\sqrt I$ of an
ideal $I$ in a ring $R$ is the intersection of all the prime ideals of $R$
containing $I$. It is known (see \cite{Larsen}) that an integral domain $D$ is an
almost Dedekind domain if and only if every non-zero ideal $I$ with its radical $\sqrt I$ a prime
ideal is a power of a prime ideal. It turns out that every Leavitt path
algebra $L$ satisfies this property. As a corollary, we show that
if $P$ is a non-zero prime ideal in a Leavitt path algebra $L$, then all the
$P$-primary ideals of $L$ form a well-ordered chain under set inclusion and
that there are no ideals of $L$ strictly between $P^{n}$ and $P^{n+1}$,
properties satisfied by Dedekind domains. We also consider the property of non-zero
ideals being cancellative, an important property that characterizes almost
Dedekind domains among integral domains. Recall that a non-zero ideal $A$ in a
ring $R$ is \textit{cancellative} if, for any two ideals $B,C$ of $R$, $AB=AC$
implies that $B=C$. Examples indicate that not all Leavitt path algebras have
this property. We show that, in a Leavitt path algebra $L:=L_{K}(E)$, every
non-zero ideal of $L$ is cancellative if and only if either (a) there is a
cycle $c$ without exits in $E$ based at a vertex $v$ such that $u\geq v$ for
every $u\in E^{0}$ and $\mathbf{H}_{E}\mathbf{=\{}H\mathbf{\}}$ where $H$ is
the hereditary saturated closure of $\{v\}$ and $B_{H}=\emptyset$, or (b) $E$
satisfies Condition (K), $|\mathbf{H}_{E}|\leq2$, for any two $X,Y\in
\mathbf{H}_{E}$ with $X\neq Y$, $X\cap Y=\emptyset$ and, for each
$H\in\mathbf{H}_{E}$, $B_{H}$ is empty and $H$ is the saturated closure of
each $u\in H$. Here $\mathbf{H}_{E}$ denotes the set of all non-empty proper
hereditary saturated subsets of vertices in the graph $E$. 
Equivalent conditions on $L$ are, either (a) $L$ contains a graded ideal $M$ which
contains every proper ideal of $L$ and $M\cong M_{\Lambda}(K[x,x^{-1}])$ where
$\Lambda$ is an arbitrary finite or infinite index set or (b) $L$ has at most
two non-zero ideals each of which is graded and is a principal ideal.

We begin by showing that every Leavitt path algebra satisfies the first
mentioned property of almost Dedekind domains.

\begin{theorem} \label{aDD} Let $I$ be a non-zero ideal of $L$. If its radical $\sqrt I$ is a prime ideal, say $P$, 
then $I=P^n$, for some $n \geq 1$.
\end{theorem}

\begin{proof} If $I$ is a graded ideal, then $I=\sqrt I$ by 
\cite[Lemma 2.1]{EER}. So if $\sqrt I$ is prime, then trivially $I$ is a prime power.
Suppose now that $I$ is a non-graded ideal such that its radical $\sqrt I$ is a prime ideal. By \cite[Theorem 4]{R-2}, 
$I=I(H,S)+{\displaystyle\sum\limits_{i\in X}}
\langle f_{i}(c_{i})\rangle $, where the $c_{i}$ are distinct cycles without exits in
$E\backslash(H,S)$, $f_{i}(x)$ are polynomials in $K[x]$ with non-zero
constant terms. Now, by \cite[Lemma 5.4]{Mulp-Ranga}, $gr(\sqrt I)=gr(I)=I(H,S)$. 
By \cite[Theorem 3.12]{R-1}, the prime ideal $\sqrt I=I(H,B_{H})+ \langle p(c)\rangle $ where $ p(x)\in
K[x]$ is an irreducible polynomial with non-zero constant term, $c$ is a cycle without exists in $E\backslash(H,B_{H})$ and, 
moreover, $E\backslash(H,B_{H})$ is downward directed. The downward directness of $E\backslash(H,B_{H})$ implies that $c$ is the only cycle without exits in $E\backslash(H,B_{H})$. Thus $I$ must be of the form $I=I(H,B_{H})+ \langle f(c)\rangle $ 
where $f(x) \in K[x]$ has its constant term non-zero.   
We claim $f(x)=p^{n}(x)$ for some integer $n\geq 1$.
Suppose, on the contrary, $f(x)=q(x)g(x)$ where $q(x)\neq p(x)$ is an irreducible
polynomial with non-zero constant term. Then 
$I=I(H,B_{H})+\langle f(c) \rangle \subseteq I(H,B_{H})+\langle q(c)\rangle $. Since $E\backslash(H,B_{H})$ is downward directed, 
$I(H,B_{H})+\langle q(c) \rangle $ is a prime ideal (\cite[Theorem 3.12]{R-1}) and so it contains the radical of $I$, namely 
$I(H,B_{H})+ \langle p(c)\rangle $. Then, in $L/I(H,B_{H})\cong L_{K}(E\backslash(H,B_{H}))$,
$\langle p(c) \rangle \subsetneq \langle q(c)\rangle \subseteq M$, the ideal generated by $c^{0}$. Now,
by \cite[Theorem 2.7.1]{AAS}, $M\cong M_{\Lambda}(K[x,x^{-1}])$ for a suitable set $\Lambda$ and by 
Proposition \ref{Ideals in Mtrix rings}, the ideal lattices of $M_{\Lambda}(K[x,x^{-1}])$ and $K[x,x^{-1}]$ are isomorphic.
Now, both $\langle p(x) \rangle $ and $ \langle q(x)\rangle $ are maximal ideals in $K[x,x^{-1}]$ and so,
identifying $M$ with $M_{\Lambda}(K[x,x^{-1}])$, both $\langle p(c) \rangle $ and $ \langle q(c)\rangle $ are maximal ideals of $M$. 
This is a contradiction since $ \langle  p(c) \rangle \subsetneq \langle q(c) \rangle $.
Hence $f(x)=p^{n}(x)$ for some $n \geq1$. Then $I=I(H,B_{H})+ \langle p^{n}
(c) \rangle =(I(H,B_{H})+ \langle p(c) \rangle)^{n}$ is a power of the prime ideal $I(H,B_{H})+  \langle p(c) \rangle$.
\end{proof}

\bigskip

Recall that, given a non-zero prime ideal $P$, an ideal $I$ is called
$P$\textit{-primary}, if its radical $\sqrt I=P$. It is known (see Theorem
6.20, \cite{Larsen}) a noetherian domain $D$ is a Dedekind domain if and only if
for any non-zero prime ideal $P$ of $D$, the set of all the $P$-primary ideals
of $D$ is a totally ordered set under inclusion and, if and only if there are
no ideals of $D$ strictly between $P$ and $P^{2}$. As an easy corollary to
Theorem \ref{aDD}, we show that a Leavitt path algebra $L$ satisfies both
these properties of a Dedekind domain.

\begin{corollary} Let $L$ be a Leavitt path algebra and let $P$ be a
non-zero prime ideal of $L$. Then

(i) \ the set of all the $P$-primary ideals is totally ordered under set inclusion;

(ii) there is no ideal $A$ of $L$ satisfying $P^{2}\subsetneq A\subsetneq P$.
\end{corollary}

\begin{proof}
(i) Note that if $I$ is a $P$-primary ideal of $L$, then
$\sqrt I=P$. By Theorem \ref{aDD}, $I=P^k$ for some $k \geq 1$. 
Thus the set of $P$-primary ideals of $L$ is the set $\{P^{n}:n\geq1\}$ which is a
totally ordered (countable) set under inclusion.

(ii) We shall actually show that there is no ideal $A$ satisfying
$P^{n+1}\subsetneq A\subsetneq P^{n}$ for any $n\geq 1$. Note that if
there is an ideal $A$ of $L$ satisfying $P^{n+1}\subseteq A\subseteq P^{n}$
for some $n\geq 1$, then $\sqrt A=P$ and so, by Theorem \ref{aDD}, $A$ is a
power of $P$ and hence $A=P^{n+1}$ or $P^{n}$. 
\end{proof}

Next, we consider the cancellative property of the ideals of an almost
Dedekind domains. This property does not seem to hold in arbitrary Leavitt
path algebras as the next two examples show.

\begin{example}
\label{CounExmp}Consider the following graph $E$:

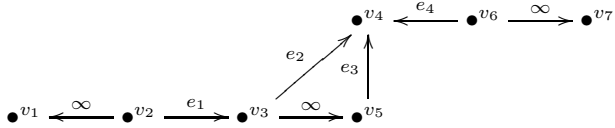
$$ \xymatrix{	&&& {\bullet}^{v_4} & {\bullet}^{v_6} \ar@{->}[l]_{e_4} 
\ar@{->}[r]^{\infty} & {\bullet}^{v_7}\\
{\bullet}^{v_1} &  {\bullet}^{v_2} \ar@{->}[l]_{\infty}    \ar@{->}[r]^{e_1} &
{\bullet}^{v_3} \ar@{->}[ur]^{e_2} \ar@{->}[r]^{\infty} & {\bullet}^{v_5} \ar@{->}[u]^{e_3}  & & }
	$$ 
	(In a graph $ \xymatrix{	 {\bullet}^{v} \ar@{->}[r]^{\infty} & {\bullet}^{w} }	$ means that  there are infinitely many edges emitted from $v$ to $w$, that is $v$ is an infinite emitter. ) 

Given the ideals $A=\langle v_1 \rangle $, $B=\langle v_5\rangle $ and $C=\langle v_7 \rangle $ of $L_{K}(E)$. Clearly $B\neq C$ but $AB=0=AC$.
\end{example}

%Remark that for any non-zero sets $A, B, C$, $ A \cap B =  A \cap C $  implies $B=C$ under the condition $A=B \cup C$.  
%In the above example $A$ is a graded ideal, so $ A \cap B = AB$ and $AC =  A \cap C $ by \cite[Lemma 3.1(i)]{Mulp-Ranga}. 
%Also, $A= \langle v_1 \rangle  \neq B \cup C$, as $v_1 \not in B \cup C$. Hence, cancellation property does not hold. 
\bigskip

\begin{example}
\label{CounExmp2} Consider the graph $E$

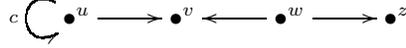
\medskip
	
	$$ \xymatrix{	{\bullet}^u  \ar@(ul,dl)_c   \ar@{->}[r]  & {\bullet}^v & {\bullet}^w \ar@{->}[l]		\ar@{->}[r] & {\bullet}^z }
	$$ 
	\medskip

Then $H = \{ v \}$ is a hereditary saturated subset. Let 
$A =\langle H \rangle$, be the principal ideal generated by $H$. 
Clearly $c$ has no exits in $E \backslash H$. Let $B$ be the nongraded ideal $A+ \langle p(c) \rangle$, where $p(x)$ is a polynomial
in $K[x]$. Clearly $gr(B) = A$. Since $A$ is a graded ideal, we apply
\cite[Lemma 3.1(i)]{Mulp-Ranga}, to conclude that $AB = A \cap B = A = A^2 = AA$. However, $A \neq B$.
\end{example}

%The next theorem gives necessary and sufficient conditions under which
%non-zero finitely generated  ideals of a Leavitt path algebra $L$ are
%cancellative. Interestingly, in this case, every non-zero ideal of $L$ is also cancellative.\\

The next theorem gives necessary and sufficient conditions (both graphical and algebraic) under which
non-zero finitely generated  ideals of $L$ is
cancellative. Interestingly, in this case, every non-zero ideal of $L$ also turns out to be cancellative.\\

We prove a useful lemma and in its proof we shall again use the result
 \cite[Lemma 3.1]{R-2} that if $A$ is a graded ideal of $L$, then for any
other ideal $B$, $A\cdot B=A\cap B$ and that $A^{2}=A$.

\begin{lemma}
	\label{Zero graded part} If the cancellation property for finitely generated
	ideals holds in $L$, then there cannot be two ideals $A,B$ with
	$A\subsetneq B$ and $A$ graded and non-zero. In particular, $gr(B)=0$.  
\end{lemma}

\begin{proof}
	By \cite{T}, $A=I(H,S)$, where $H=A\cap E^{0}$ and $S\subseteq B_{H}$. Since
	$A\neq0$, $H\neq\emptyset$. Then, for a vertex $u\in H$, $C=\left\langle
	u\right\rangle$ is a finitely
	generated graded ideal. If $ B $ is an ideal such that $A\subsetneq B$ then we get $C\cdot B=C\cap B=C=C\cdot C$, but $B\neq
	C$, a contradiction to the cancellation property. This proves the first
	statement. By taking $A=gr(B)$, the second statement follows from the first. 
\end{proof}

\begin{corollary}
	\label{closure of a single vertex} If the cancellation property for finitely
	generated ideals holds in $L$, then every non-empty hereditary saturated
	subset $H\subsetneq E^{0}$ is the hereditary saturated closure of each
	single vertex $u\in H$ and, moreover, $B_{H}=\emptyset$. In particular, every
	non-zero graded ideal of $L$ is a principal ideal, being generated by a single vertex.
\end{corollary}

\begin{proof}
	If there is a vertex $u\in H$ such that the hereditary saturated closure $X$
	of $\{u\}$ is not equal to $H$, then the non-zero graded ideal $A=\langle X \rangle$
	satisfies $A\subsetneq \langle H \rangle $, contradicting Lemma \ref{Zero graded part}.
	Likewise, if $B_{H}\neq\emptyset$, then again we have the non-zero graded
	ideal $I(H,\emptyset)\subsetneq I(H,B_{H})$, contradiction Lemma
	\ref{Zero graded part}. If $A=I(H,S)$ is a non-zero graded ideal of $L$, then
	since $B_{H}=\emptyset$, $S=\emptyset$ and since $H$ is the hereditary saturated closure
	of any vertex $u\in H$, $A=\langle H \rangle=\left\langle \{u\}\right\rangle$ is a principal ideal.
\end{proof}

\begin{theorem}
	\label{Cancellation Thm}Let $E$ be an arbitrary graph. The following
	conditions are equivalent for $L$:
	
	\begin{description}
		\item[(i)] The cancellation property holds for all non-zero ideals in $L$;
		
		\item[(ii)] The cancellation property holds for all non-zero finitely
		generated ideals of $L$;
		
		\item[(iii)] Either (a) there is a cycle $c$ without exits in $E$ based at
		a vertex $v$ such that $u\geq v$ for every $u\in E^{0}$ and 
		$\mathbf{H}_{E}\mathbf{=\{}H\mathbf{\}}$ where $H$ is the hereditary 
		saturated closure of $\{v\}$ and $B_{H}=\emptyset$, or 
		(b) $E$ satisfies Condition (K),
		$|\mathbf{H}_{E}|\leq2$, for any two $X,Y\in\mathbf{H}_{E}$ with $X\neq Y$,
		$X\cap Y=\emptyset$ and, for each $H\in\mathbf{H}_{E}$, $B_{H}$ is empty and
		$H$ is the saturated closure of each $u\in H$.
		
		\item[(iv)] Either (a) $L$ contains a graded ideal $M$ which contains every
		proper ideal of $L$ and $M\cong M_{\Lambda}(K[x,x^{-1}])$ where $\Lambda$ is
		an arbitrary finite or infinite index set or (b) $L$ has at most two non-zero
		ideals each of which is graded and is a principal ideal and is a principal ideal.
		%$A\cap B=0$ for any two distinct ideals $A,B$.
	\end{description}
\end{theorem}

\begin{proof}
	Clearly (i)$\Rightarrow$(ii).

	Assume (ii). Case (a): Suppose $L$ has a non-graded prime ideal $P$ with
	$P\cap E^{0}=H^{\prime}$. By  \cite[Theorem 3.12]{R-1}, $P=I(H^{\prime
	},B_{H^{\prime}})+\langle p(c) \rangle$, where $c$ is a cycle without exits based at a vertex
	$v$ in $E\backslash(H^{\prime},B_{H}^{\prime})$, $p(x)$ is an irreducible
	polynomial in $K[x,x^{-1}]$ such that $p(c)\in P$ and $u\geq v$ for all $u\in
	E^{0}\backslash H^{\prime}$. By Lemma \ref{Zero graded part}, $I(H^{\prime
	},B_{H^{\prime}})=0$ and hence both $H^{\prime}$ and $B_{H^{\prime}}$ are
	empty. This means that $E$ contains a unique cycle $c$ without exits based
	at a vertex $v$ and $u\geq v$ for every vertex $u\in E^{0}$. Let $H$ be the
	hereditary saturated closure of $\{v\}$. Observe that there cannot be two
	members $X,Y\in\mathbf{H}_{E}$ with one of them properly containing the other, say
	$X\subsetneq Y$. Because, we will then have two non-zero graded ideals $A=\langle X \rangle$
	and $B=\langle Y \rangle$ with $A\subsetneq B$ and this is not possible by Lemma 	\ref{Zero graded part}. 
	We claim that $\mathbf{H}_{E}\mathbf{=\{}H\mathbf{\}}$. Suppose, on the contrary, 
	there is another element $Z\in\mathbf{H}_{E}$. As
	noted above, $Z\nsubseteq H$ and $H\nsubseteq Z$. So $Z\cap H\subsetneq
	H$ and is further non-empty since $v\in Z\cap H$. This then gives rise to
	two\ non-empty hereditary saturated subsets with one containing the other, a
	contradiction. Also $B_{H}=\emptyset$, by Corollary
	\ref{closure of a single vertex}. This proves (iii)(a).
	
	Case (b): Suppose every prime ideal of $L$ is graded. Then, by \cite[Corollary 3.13]{R-1}, 
	$E$ satisfies Condition (K) and so every ideal of $L$ is
	graded.  If $|\mathbf{H}_{E}|=0$, then $\mathbf{H}_{E}=\emptyset$ and we are
	done. Assume $|\mathbf{H}_{E}|\neq0$. Suppose $X,Y\in\mathbf{H}_{E}$ with
	$X\neq Y$. We claim $X\cap Y=\emptyset$. Because if a vertex $u\in X\cap Y$,
	then, by Corollary \ref{closure of a single vertex}, both $X$ and $Y$ are
	saturated closures of $\{u\}$ and hence $X=Y$, a contradiction.  We claim that
	$|\mathbf{H}_{E}|\leq2$. Indeed if there are three distinct members
	$X,Y,Z\in\mathbf{H}_{E}$, then the three distinct non-zero graded ideals
$A=\langle X \rangle,B=\langle Y \rangle,C=\langle Z \rangle$ are principal by Corollary
\ref{closure of a single vertex} and satisfy $A\cdot B=A\cap B=0=A\cap
C=A\cdot C$, but $B\neq C$. This contradiction shows that $|\mathbf{H}%
_{E}|\leq2$. By Corollary
	\ref{closure of a single vertex}, each $H\in\mathbf{H}_{E}$ is the hereditary
	saturated closure of each $u\in H$ and the corresponding $B_{H}=\emptyset$.
	This proves (iii)(b).
	
	Assume (iii) (a). Let $M$ be the ideal of $L$ generated by the hereditary saturated closure $H$ of $\{v\}$,	
	so $M$ is a graded ideal of $L$.  Since $c$ is a
	cycle without exits, by \cite[Lemma 2.7.1]{AAS}, $M\cong M_{\Lambda
	}(K[x,x^{-1}])$ where $\Lambda$ is an arbitrary finite or infinite index,
	being the set of all paths in $E$ that end at $v$, but do not include the
	entire cycle $c$. Since $\mathbf{H}_{E}=\{H\}$ and since $B_{H}=\emptyset$,
	$M$ is the only non-zero graded ideal of $L$. We claim that every non-zero
	ideal $N$ of $L$ is contained in $M$ (and non-graded). Suppose $N\neq M$. Now
	$N$ must be a non-graded ideal, since $M$ is the only non-zero graded ideal of
	$L$. \ Then, by Lemma \ref{Zero graded part}, $gr(N)=0$ and so $N$ does not
	contain any vertices. As $E^{0}$ is downward directed,  \cite[Lemma 3.5]{R-1}
	then implies that $N=\langle f(c^{\prime}) \rangle$ where $f(x)\in K[x]$ and $c^{\prime}$ is
	a cycle without exits in $E$. Since $u\geq v$ for every $u\in E^{0}$, $c$ is
	the only cycle without exits in $E$ and so $c^{\prime}=c$. Then
	$N=\langle f(c)\rangle \subset M$, as $c\in M$. This proves (iv)(a).
	
	Assume (iii)(b). Since Condition (K) holds, every ideal of $L$ is graded \cite{AAS}. If
	$\mathbf{H}_{E}\mathbf{=\emptyset}$, then the only hereditary saturated
	subsets of $E^{0}$ are $E^{0}$ and $\emptyset$ and so $L$ contains no non-zero
	proper ideals. Suppose $\mathbf{H}_{E}\mathbf{=\{}H\mathbf{\}}$, with
	$B_{H}=\emptyset$. Then $I=\langle H,\emptyset\rangle$  is the only proper non-zero ideal
	of $L$. Suppose $\mathbf{H}_{E}\mathbf{=\{}H_{1},H_{2}\mathbf{\}}$. Since $B_{H_{1}}=\emptyset=B_{H_{2}}$, $A=\langle H_{1} \rangle $ and $B=\langle H_{2}\rangle $ are the only proper non-zero ideals of $L$ and they both are principal ideals as
 $H_{1},H_{2}$ are saturated closures of single vertices $v_{1}\in H_{1}$ and $v_{2}\in H_{2}$. This proves (iv)(b).\\
	
	Assume (iv)(a) so that $L$ contains a graded ideal $M\cong M_{\Lambda}(K[x,x^{-1}])$ where $\Lambda$ is an arbitrary finite or infinite index set and that $M$ contains every other proper ideal of $L$. 	
Now, by Proposition 1,
the ideals of $M_{\Lambda}(K[x,x^{-1}])$ are of the form $M_{\Lambda}(I)$
where $I$ are the ideals of $K[x,x^{-1}]$, the map $I\longmapsto M_{\Lambda
}(I)$ is an isomorphism of the ideal lattices of $K[x,x^{-1}]$ and
$M_{\Lambda}(K[x,x^{-1}])$ and, further, $M_{\Lambda}(I\cdot J)=M_{\Lambda
}(I)\cdot$ $M_{\Lambda}(J)$ for any two ideals of $K[x,x^{-1}]$. Since the
cancellation for non-zero ideals hold in the principal ideal domain
$K[x,x^{-1}]$, we conclude that the cancellation property holds for non-zero
ideals in $M$. Now the graded ideal $M$ possesses local units, being
isomorphic to a Leavitt path algebra of a suitable graph (see \cite{T2}). From
this, it easy to show that the ideals of $M$ are also the ideals of $L$. Since
$M$ contains every other ideal $A$ of $L$ and is graded (so $MA=M\cap A$), we
conclude that $M$ is also cancellative. Thus all the non-zero ideals of $L$
have the cancellation property. This proves (i).

Now (iv)(b)$\Rightarrow$
(i) is immediate since $L$ has at most two distinct non-zero proper ideals
which are graded and thus the cancellation property holds trivially in $L$.
\end{proof}

We conclude the paper by illustrating some examples of the graphs that satisfy 
the conditions of part (iii) of Theorem \ref{Cancellation Thm}. 
\begin{example}
\label{CancelExmp1} \rm Consider the graph $E_1$ 

\medskip
	
	$$ 
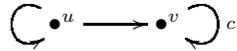
\xymatrix{	{\bullet}^u  \ar@(ul,dl)  \ar@{->}[r]  & {\bullet}^v \ar@(ur,dr)^c  }
	$$ 
	\medskip

which satisfies the conditions of part (iii) a). So,  
$H = \{ v \}$ is the only non-zero proper hereditary saturated subset and 
$M =\langle H \rangle \cong K[x,x^{-1}]$ contains every proper ideal of $L_K(E_1)$.

\medskip

Consider the graph $R_2$ 

\medskip
	
	$$ 
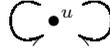
\xymatrix{	{\bullet}^u  \ar@(ul,dl) \ar@(ur,dr) } 	$$ 
	
	\medskip

which satisfies the conditions of part (iii) b) and   
there are no non-zero proper hereditary saturated subsets.   
Actually $L_K(R_2)$ is the simple Leavitt algebra of type (1,2).

\medskip

Consider the graph $E_3$ 

\medskip
	
	$$ 
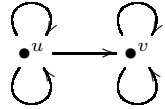
\xymatrix{	{\bullet}^u   \ar@(ul,ur) \ar@(dl,dr) \ar@{->}[r]  & {\bullet}^v \ar@(ul,ur) \ar@(dl,dr) }
	$$ 
	\medskip

which satisfies the conditions of part (iii) b) and   
$H = \{ v \}$ is the only non-zero proper hereditary saturated subset.  
%$L_K(E_3)$

Consider the graph $F$ 

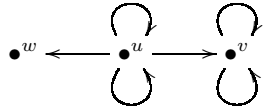
\medskip

	$$ \xymatrix{	{\bullet}^w  & \ar@{->}[l]   {\bullet}^u  \ar@(ul,ur) \ar@(dl,dr) \ar@{->}[r]  & {\bullet}^v \ar@(ul,ur) \ar@(dl,dr) }$$

\bigskip

which satisfies the conditions of part (iii) b) and   
$H_1 = \{ v \}$ and $H_2 = \{ w \}$ are the only two non-zero proper hereditary saturated subsets. 

\end{example}

{\bf Acknowledgement} \\
%\begin{acknowledgement}
The first three authors are deeply thankful to Nesin Math Village, \c{S}irince, Izmir for providing an excellent research environment where this work has been developed.

\bigskip

Song\"{u}l Esin 
	\\19 May\i s mah. No:10A/25 Kad\i k\"{o}y, \.{I}stanbul, Turkey. \texttt{songulesin@gmail.com}\\
	
M\"{u}ge Kanuni
\\ Department of Mathematics, D\"{u}zce University, Konuralp 81620 D\"{u}zce, Turkey. \texttt{mugekanuni@duzce.edu.tr}\\

Ayten Ko\c{c} 
\\ Department of Mathematics and Computer Science, \.{I}stanbul K\"{u}lt\"{u}r University,  \.{I}stanbul, Turkey. \texttt{akoc@iku.edu.tr}\\

Katherine Radler
\\ Department of Mathematics, Saint Louis University, St. Louis, Missouri, 63103, USA. \texttt{katie.radler@slu.edu} \\

Kulumani M. Rangaswamy
\\ Department of Mathematics, University of Colorado. Colorado Springs, CO 80918, USA. \texttt{krangasw@uccs.edu}


\begin{thebibliography}{9}

\bibitem {AAS}G. Abrams, P. Ara and M. Siles Molina, Leavitt path algebras, Lecture Notes in Mathematics, vol. 2191, Spinger (2017).

\bibitem {AF} F.W. Anderson and K.R. Fuller, Rings and categories of modules,
Graduate Texts in Mathematics, vol. 13, Springer-Verlag, NewYork - Heidelberg
- Berlin (1974).

\bibitem {EER}S. Esin, M. Kanuni and K.M. Rangaswamy, On intersections
of two-sided ideals of Leavitt path algebras, J. Pure and Applied Algebra,
vol. 221 (2017), 632 - 644.

%\bibitem{Fuchs1} L. Fuchs, \"{U}ber die Ideale Arithmetischer Ringe,
%Comment. Math. Helv. 23 (1949) 334-341.



\bibitem {L} T.Y. Lam, A First Course in Noncommutative Rings, Grad. Texts in Math.,
vol.131, Springer-Verlag, (2001).

\bibitem{Larsen} M.D. Larsen, P.J. McCarthy, Multiplicative Theory of
Ideals, Academic Press, New York-London,1971.

\bibitem {NO}C. Nastasescu, F. van Oystaeyen, Graded Ring Theory,
North-Holland, Amsterdam, (1982).

\bibitem{R-1} K.M. Rangaswamy, The theory of prime ideals of Leavitt path
algebras over arbitrary graphs, J. Algebra \textbf{375} (2013), 73 - 96.

\bibitem{R-2} K.M. Rangaswamy, On generators of two-sided ideals of Leavitt
path algebras over arbitrary graphs, Communications in Algebra \textbf{42}
(2014), 2859 - 2868. 

\bibitem{Mulp-Ranga} K.M. Rangaswamy, The multiplicative ideal theory of
Leavitt path algebras, Journal of Algebra 487 (2017) 173 - 199.

\bibitem {T} M. Tomforde, Uniqueness theorems and ideal structure of Leavitt path algebras, J. Algebra 318 (2007) 270 - 299.

\bibitem {T2} E. Ruiz, M. Tomforde, Ideals in graph algebras, Algebr. Represent. Theory 17 (2014) 849-861. 

\bibitem {ZS}O. Zariski, P. Samuel, Commutative algebra, vol. 1, van Nostrand
Company (1969).

\end{thebibliography}
\end{document}